\documentclass[12pt]{article}

	  	  \usepackage{amsmath, amsthm, amsfonts, amssymb,palatino}
	  \usepackage[all]{xy}
	  \usepackage{graphics,fullpage}
 \usepackage[usenames,dvipsnames]{pstricks}
 \usepackage{epsfig}
 \usepackage{mathabx} 
 \usepackage{comment}
 \usepackage{listings}
 \usepackage[scaled]{beramono} 
 \usepackage[T1]{fontenc}
\AtBeginDocument{\addtocontents{toc}{\protect\thispagestyle{empty}}} 
 \usepackage{hyperref}

                \def\g{\gamma}				 

  \def\O{{\mathcal O}} 
   
   \def\X{{\mathcal X}}
 \def\Z{{\mathcal Z}}

		\def\ZZ{\mathbb{Z}}

		\def\xto#1{\xrightarrow{#1}}

\newcommand{\ol}{\overline}
\newcommand{\mbf}{\mathbf}

		\newcommand{\Der}{Der}

		\newcommand{\ot}{\otimes}

		\newcommand{\om}{\omega}
		\newcommand{\Om}{\Omega}

		\def\id{\mathrm{Id}}

		\def\lim{\mathrm{lim}}

		\def\Hom{\mathrm{Hom}}
		\def\Ext{\mathrm{Ext}}

		\def\Cacti{\mathcal{C}acti}
		
		\newcommand{\lie}{\mathfrak}

		\def\!{\text{!`}}

		\newcommand{\To}{\Rightarrow}

		\newcommand{\ra}{\overline}

		\newcommand{\wh}{\widehat}
		\newcommand{\wt}{\widetilde}

		\def\End{\mathrm{End}}

		\def\Ker{\mathrm{Ker}}
		\def\mfk{\mathfrak}
		
	\theoremstyle{plain}
		\newtheorem{teo}{Teorema}[section]
		
		\newtheorem{coro}[teo]{Corollary}
		\newtheorem{lema}[teo]{Lemma}
	\theoremstyle{definition}
		\newtheorem{defi}[teo]{Definition}
		
        	\newtheorem{rmk}[teo]{Remark}
		\newtheorem{example}[teo]{Example}

	\theoremstyle{remark}

\def\udu{\raisebox{-3pt}{\includegraphics[scale=0.6]{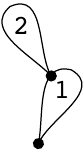}}}
\def\ud{\raisebox{-3pt}{\includegraphics[scale=0.6]{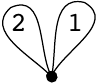}}}
\def\du{\raisebox{-3pt}{\includegraphics[scale=0.6]{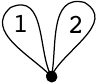}}}

\def\utud{\raisebox{-5pt}{\includegraphics{m3-a.pdf}}} 
\def\dutu{\raisebox{-5pt}{\includegraphics{m3-b.pdf}}} 

\def\ca#1#2#3{\raisebox{#3}{\includegraphics[scale=#2]{ca#1.pdf}}}

\def\m3{ \utud - \dutu }

\def\bideg{\mathrm{bideg}}
\def\ZZ{\mathbb Z}
\title{A Cacti theoretical interpretation of the axioms of bialgebras and $H$-
module algebras}

\author{Marco A. Farinati\thanks{
Member of CONICET. Partially supported by UBACyT X051, PICT
2006-00836. mfarinat@dm.uba.ar} and  Leandro E. Lombardi\thanks{
Partially supported by UBACyT X051, PICT 2006-00836. llombard@dm.uba.ar}} 

\def\g{\mathfrak g}

\begin{document}
\maketitle

\begin{abstract}
We establish a dictionary between the Cacti algebra axioms
on a Cacti algebra structure with underlying free
associative algebra, under suitable good behavior
with degrees. Using these ideas, for an associative algebra $A$ and
a bialgebra $H$,  we also translate
Cacti algebra maps $\Om(H)\to C^\bullet(A)$ (where
$\Om(H)$ stands for the cobar construction on $H$ and $C^\bullet(A)$
is the Hochschild cohomology complex)
with $H$-module algebra structures on $A$, and illustrate
with examples of applications.
\end{abstract}

\section*{Introduction and preliminaries}

In \cite{Kad05}, the author defines a Cacti algebra
structure on $\Om(H)$, the cobar construction of a d.g. bialgebra $H$.
Recall that $\Om(H)=TV$ the tensor algebra on $V=\Ker\epsilon$,
with differential of the form $d_i+d_{\Delta}$. That is, one differential
 coming from the original differential on the d.g. bialgebra $H$
and a second one coming from its coalgebra structure.
In the mentioned article, the author works over $\ZZ/2\ZZ$. In \cite{justin},
signs are introduced for any characteristic. This construction
give examples of Cacti algebras of special type, they are not only
graded but naturally bigraded, and operations have extra
properties with respect to this bigrading. We call these 
properties to be {\em well graded}
(see definition \ref{def_cacti_compatible_con_el_grado}).
We prove a kind converse of this construction that include the characterization of
the image of the functor $\Om:d.g.bialg\to Cacti$-$alg$. More precisely,
we prove that if a Cacti algebra $T$ is well graded and
freely generated as associative algebra by elements of (external) degree one,
namely $T\cong TV$ as associative algebras ($TV$=the tensor algebra on a graded vector space),
then the Cacti algebra structure on $T$ determines uniquely a d.g. bialgebra
structure on $H=V\oplus k1_H$, and hence $T=\Om(H)$ for a uniquely
determined d.g.bialgebra $H$.

The examples arising from the Kadeishvili construction
are not the only well graded ones.
The historically most important
family of Cacti algebras, namely the Hochschild complex $C^\bullet(A)$
of an associative (eventually d.g.) algebra $A$ is also a well
 graded Cacti algebra.
In lemma~\ref{lemamorfismo}), we study morphisms between well graded Cacti algebras.
A consequence of this result can be seen as a continuation of
 the dictionary
between Cacti algebras and bialgebras. More precisely (Theorem \ref{teomorfismo})
we prove that, given an (eventually d.g.) bialgebra $H$ and associative algebra $A$,
the set of morphism between bigraded Cacti algebras $\{\Om(H)\to
C^\bullet(A)\}$ is in 1-1 correspondence with structures of $H$-module algebra on $A$.

We end with examples of applications to the Gerstenhaber algebra structure on the
Hochschild cohomoly of an algebra.

For the purpose of this work, a Cacti algebra in al algebra over de operad $\X_2$ defined in~\cite{BF04}.
This operad (up to sign a convention) is called $S_2$ in~\cite{McS02}, spineless cacti in~\cite{Kau07} and
is also the operad codifying Gerstenhaber-Voronov algebras~\cite{GV95}. We briefly recall the definiton: a {\em Cacti algebra}  is a 
differential graded vector space $(T,d)$ with operations

\begin{enumerate}
 \item  $C_2:T\ot T\to T$, an associative product.
 \item for any $n\geq 2$, $B_n:T^{\ot n}\to T$ are brace operations.
 \end{enumerate}
satisfying a set of compatibility relations that we list below. 
In order to write them, it is convenient to
use a graphic representation:
\[
C_2 = \ud, \quad \quad B_2 = \udu, \quad \quad B_n =  \ca{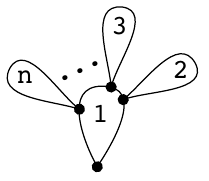}{1}{-10pt}
\]
For example, the brace relations   can be described pictorically as
\begin{eqnarray*}
B_n \circ_1 B_m & = & 
\raisebox{-5pt}{\includegraphics[scale=0.7]{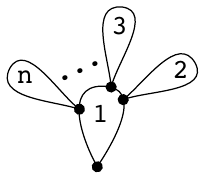}} \circ_1 \raisebox{-5pt}{\includegraphics[scale=0.7]{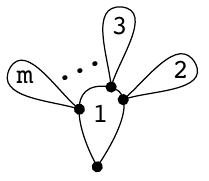}} \\ 
& = & \sum_\text{possibilities} \pm \raisebox{-10pt}{\includegraphics[scale=0.7]{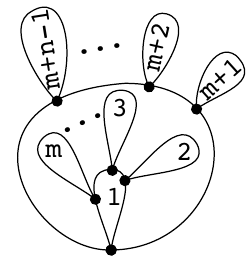}}  \\
\end{eqnarray*}
where the sign is given by the permutation
of the dots belonging to $B_n$ and $B_m$.

\

The distributivity law betwee  $C_2$ and  $B_m$ is:
\[
\raisebox{-5pt}{\includegraphics[scale=0.7]{Bm.pdf}} \circ _1  \raisebox{-3pt}{\includegraphics[scale=0.7]{ca12.pdf}} = \sum_k \raisebox{-5pt}{\includegraphics[scale=0.7]{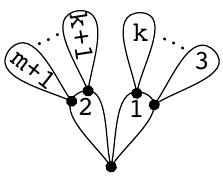}}
\]

And finally, the relation with the differential is
$\partial C_2=0$ and

\[
\partial B_m=\partial\big(
\raisebox{-8pt}{\includegraphics[scale=0.7]{Bm.pdf}} \big)
=  \raisebox{-8pt}{\includegraphics[scale=0.7]{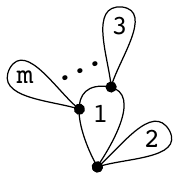}}
+ \sum_{i=2}^{m-1} (-1)^{i+1} \raisebox{-8pt}{\includegraphics[scale=0.7]{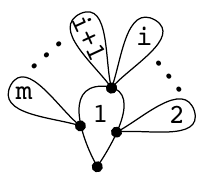}}
+(-1)^{m+1}  \raisebox{-8pt}{\includegraphics[scale=0.7]{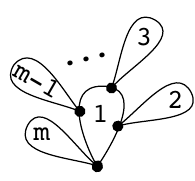}}
\]
   where, if $P:T^\ot n\to T$ is an operation,
 $\partial P$ is by definition the operation given by
  \[
  (\delta P)(t_1\ot \cdots \ot t_n)=d(P(t_1\ot \cdots \ot t_n))-
  \sum_{i=1}^n (-1)^{|P|+\sum\limits_{j=1}^{i-1} |t_j|}
  P(t_1\ot \cdots \ot d(t_i)\ot\cdots \ot t_n)
\]
In particular, $\partial C_2=0$ means that in any Cacti algebra, the differential
is a derivation for the product  $C_2$.

%

\section{$\Cacti$-algebra structure in $TV$}

Let $V$ be a graded vector space, then
$TV=\oplus_{n\geq 0}V^{\ot n}$ is a free associative algebra, and it
is bigraded taking, for $v_k\in V_{i_k}$
\[
\bideg (v_1\ot\cdots\ot v_n)=(\sum_{i=1}^n|v_i|_V,n)
\]
we call $\sum_{i=1}^n|v_i|_V$ the {\em internal}
degree, and $n$ the {\em external} or tensorial degree.
We remark that the total degree
\[
|v_1\ot\cdots\ot v_n|_{tot}=(\sum_{i=1}^n|v_i|_V)+ n
\]
is the same as the usual degree on the tensor
algebra of $\Sigma V$,
the suspension of $V$. This total degree is most usually considered,
but we prefer to keep the information of the bigrading by reasons that will be
clear in the rest of this work.

\begin{rmk}\label{dDeltaprima} Let $V$ be a trivially graded vector space (i.e. $V=V_0$), then the data of a
a square zero differential in $A=TV$ of total degree one is equivalent
to give a (non necessary counital)
coassociative coalgebra structure in $V$.

If $V$ is arbitrarily graded, then the data of a square zero differential in
$TV$ is equivalent to a differential in $V$ together with
an up-to homotopy coassociative coalgebra structure in $TV$, but if
the differential in $TV$ is of the form $D = d_i + d_e$
where $\bideg d_i=(1,0)$ and $\bideg d_e=(0,1)$ then
to give $D$ is equivalent to a  (strict) coassociative
differential coalgebra structure in $V$. Take simply $d_V=d_i|_V$, and
$\Delta' =d_e|_{V}$.
\end{rmk}

\begin{rmk}
 \label{DeltaprimaDelta}
The non necessarily counital coassociative structures in $V$
are in 1-1 correspondence with the unital coassociative
structures in
$H:=V\oplus k 1_H$, where $1_H$ is a new formal element satisfying
$\Delta(1_H) = 1_H \ot 1_H$. The correspondence is given
by
$\Delta'\leftrightarrow \Delta$ with

\[
\Delta:H\to H\ot H
\]
\[
\Delta(v):=\Delta'(v)+1_H\ot v+v\ot 1_H
=v_{1}\ot v_{2}+1_H\ot v+v\ot 1_H\]
for $v\in V$,  $\Delta 1_H:=1_H\ot 1_H$. And given
 $\Delta:H\to H\ot H$, let $\pi:H\to V$ be the canonical projection
 with respect to the direct sum decomposition $H=V\oplus k1_H$, then
 \[
 \Delta' (v):=(\pi\ot\pi)\circ\Delta
 \]
 The counit in $H$ is given by $\epsilon(v)=0$ if $v\in V$ and
 $\epsilon(1_H)=1$. Working with elements one can easily see that the coassociativ equation for $\Delta$ and $\Delta' $ is {\em the same},
 so $\Delta$ is coassociative iff $\Delta' $ is, and letting
 $1_H$ having internal degree 0 (but tensorial degree 1), and
 $d_i(1_H)=0$ the correspondence works equally well for the graded case
 \end{rmk}
 
We will  consider Cacti-algebra structures on $TV$ of a certain type.
Recall that the cactus $C_2=\ud$ provides a strict associative product.
We will say that
the cacti algebra structure on $TV$ {\em extends} the one in $TV$
if $\ud(x,y)=x\ot y$ (where $x,y\in TV$).
Notice that in $\ra T V$, this property implies that
every element of $A$ can be obtained
from $V$ and the action of the cactus $C_n$, namely
if   $\mbf x = x_1 \ot \dots \ot x_n$ then
 $\mbf x = C_n(x_1, \dots, x_n)$.

Next definition is motivated by the example of the Hochschild complex
 $C^{\bullet}(A)$ of an associative algebra $A$.
Recall that in
$C^{\bullet}(A)$, if
$f:A^{\ot n}\to A$, then the brace operation is a formula of type
\[
f \{ g_1, \dots, g_k\} = \sum \pm f(\cdots,g_1(-),\cdots,g_2(-),\cdots)
\]
and this implicitely says that if $
n<k$ then
\[
f \{ g_1, \dots, g_k\} = 0
\]
These brace operations corresponds to the cactus
\[
B_{k+1}  =  \ca{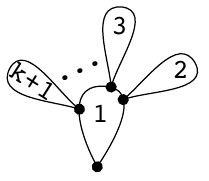}{1}{-10pt}
\]

\begin{defi}
\label{compdg} \label{def_cacti_compatible_con_el_grado}
Let $C$ be bigraded  vector space
\[
C = \bigoplus_{p,q} C^{p,q}
\]
with a Cacti algebra structure on it with respect to the total degree.
We will say that this structure is {\em well graded} if
\[a \in C^{ \bullet,p}, p < n-1  \implies B_n(a, \dots) = 0 \]
and the differential is compatible with the bigrading in the sense that
$d=d_i+d_e$ where
\[
\begin{array}{ccccc}
d_i & :  & T^{n,\bullet} & \to & T^{n+1,\bullet} \\
d_e & :  & T^{\bullet, n} & \to & T^{\bullet, n+1} \\
\end{array}
\]
Moreover, we ask  $C_2$ and $B_m$ ($m\geq 2$) to be homogeneous
 with respect to the internal degree.
\end{defi}

Notice that if $C$ is a cactus algebra that is graded (and not bigraded),
then it can be considered as trivially bigraded with $C^{0,q}=C^q$
and $C^{p,q}=0$ for $p\neq 0$, and the definition of well graded
makes sense.

\begin{example} The Hochshild complex of an associative
algebra is a well graded Cacti algebra, this example is trivially bigraded.
But also if $A$ is a differential graded associative algebra, then
$C(A)$ is well graded. In both cases, the bidegree is given by
\[
C^{p,q}(A)=\Hom(A^{\ot q},A)_p
\]
where $\Hom(-,-)_p$ is the set of homogeneous linear transformations
of degree $p$ (between two graded vector spaces).  
\end{example}

\begin{example} If $(H,*,\Delta,d)$
 is a differential graded associative bialgebra,
then in particular it is a differential graded coalgebra, and
the cobar construction makes sense
\[
\Om(H)=(TV,d)
\]
where $V=\ra H=\Ker\epsilon$ and $d=d_H+d_{\Delta}$.
In \cite{Kad05}, Kadeishvili exhibits (in characteristic 2)
a Cacti-algebra structure
on $\Om(H)$ coming from the bialgebra structure of $H$. In
\cite{justin} the author introduce appropriate signs showing that $\Om(H)$
 is a Cacti algebra in any characteristic (e.g 0). In this construction, the
 brace structure is given by
 \[
B_m (\mbf x,\ol{ \mbf y} )
:=
\hskip -0.5cm
\sum_{1\le i_1 < \ldots < i_{m-1} \le n}
\hskip -0.5cm
\pm x_1 \otimes \ldots \otimes (x_{i_1} * \mbf y_1) \otimes \ldots \otimes (x_{i_{m-1}} * \mbf y_{m-1}) \ot \ldots \ot  x_n
\]
where in each term, the sign is the Koszul-permutation sign of the symbols

\mbox{$* \ldots * x_1 \dots x_n \mbf y_1 \dots \mbf y_{m-1} \mapsto
x_1 \dots x_{i_1} * \mbf y_1 x_{i_1+1} \dots x_{i_{m-1}} * \mbf y_{m-1} \dots x_n$}. \\

and the notation is $\mbf x=x_1\ot\cdots\ot x_n$, and
$\ra{\mbf y}=(\mbf
y_1\dots,\mbf y_{m-1})$.
We remark that here, for $x\in H$, its symbol has
degree $|x|_{tot}=|x|_H+1$, and if
$\mbf y=y_1\ot\cdots\ot y_n$, then its symbol has
degree $n+\sum_{i=1}^n|y_i|_H$. \\

In this formula, it is implicitly assumed that $m-1\leq n$,
otherwise it is zero, so this is also an example of well-graded
Cacti algebra.
\end{example}

\begin{defi}
Let $C$ be an arbitrary Cacti algebra and let us denote
 $*$ the operation induced by  $B_2 = \udu $ in $C$.
 More precisely,
 $$a * b := (-1)^{|a|} \udu (a,b) = (-1)^{|a|} B_2 (a,b)$$
\end{defi}

This product is pre-Lie by definition. In well graded Cacti algebras, it
is associative when restricted to external degree one, as we can see in the following lemma.

\begin{lema} \label{lema_pre-Lie_es_asoc}
Let $C$ be a well graded Cacti algebra and set
$C^1=\oplus_p C^{p,1}$ the subspace of elements of external
degree one. Then, for $y,z\in C$ and $x\in C^1$
\[
(x*y)*z= x*(y*z)
\]
Notice that $C^n*C^m\subseteq C^{n+m-1}$ so, in articular,
$(C^1,*)$ is a (non necessarily unital) associative algebra
and $C^n$ is a $C^1$-module.
  \end{lema}

\begin{proof} We will compute the associator and see that it
is governed by $B_3$, which by hypothesis is zero when
the first variable belongs to $C^1$:

Let $x,y,z \in C$ with $x\in C^1$, we have
\begin{eqnarray*} 
(x*y)*z - x*(y*z)  & =  & (-1)^{|y|+1} \udu(\udu(x,y),z) - (-1)^{|x|+|y|}\udu(x,\udu(y,z))  \\
& =  &  (-1)^{|y|+1} \Big( \udu(\udu(\; ,\; ),\; ) + \udu(\;,\udu(\;,\;)) \Big) (x,y,z) \\
& = &  (-1)^{|y|+1} \left( \udu \circ_1 \udu + \udu \circ_2 \udu \right) (x,y,z)\\
& = &  (-1)^{|y|+1} \left(
\raisebox{-5pt}{\includegraphics[scale=0.5]{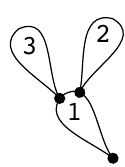}}
+ \raisebox{-5pt}{\includegraphics[scale=0.5]{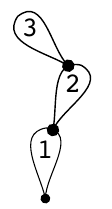}}
- \raisebox{-5pt}{\includegraphics[scale=0.5]{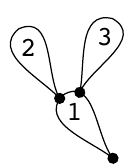}}
- \raisebox{-5pt}{\includegraphics[scale=0.5]{ca12321.pdf}} \right) (x,y,z) \\
& = &  (-1)^{|y|+1} \left(
\raisebox{-5pt}{\includegraphics[scale=0.5]{ca12131.pdf}}
- \raisebox{-5pt}{\includegraphics[scale=0.5]{ca13121.pdf}}
\right) (x,y,z) \\
\end{eqnarray*}
(Signs are due to Koszul rule for the total degree of the symbols
$x,y,z \in C$ and $B_2$.)\\

Because we assume $C$ is well graded, the cactus
\raisebox{-2pt}{\includegraphics[scale=0.5]{ca12131.pdf}}   and
 \raisebox{-2pt}{\includegraphics[scale=0.5]{ca13121.pdf}}
act trivially when the first variable is in $C^1$, so the associator
vanishes.
\end{proof}

\begin{coro}
Let $V$ be a graded vector space and suppose a well graded Cacti algebra
 structure is given in $\ra T V$, then this structure induce by restriction
an associative product $*:V\times V\to V$.
\end{coro}

From now on we concentrate in the bigraded associative algebra
$TV$, and we will consider
all possible well-graded Cacti algebra structures on it.
We recall that the external degree is the tensorial degree, and hence
a $d$-dimensional cactus acts as an operation of (external)
 degree $-d$,
and the differential is of total degree one.

\begin{rmk}
A (non necessarily unitary) operation $*:V\times V\to V$ can
be extended to $H:=V \oplus k1_H$ declaring $1_H$ as formal unity for $*$, namely
\[
1_H*v:=v=:v*1_H\ (\forall v\in V) \quad \text{y} \quad 1_H*1_H:=1_H
\]
Notice that  $*$ es associative in $V$ if and only if  it is associative in $H$.

Recall that a (well graded) differential in $TV$ induces (by restriction to $V$ a coassociative and counitary comultiplication in $H$ via
\begin{eqnarray*}
\Delta 1_H & = & 1_H \otimes 1_H \\
\Delta v & = & d_e(v)+v \otimes 1_H + 1_H \otimes v \\
\end{eqnarray*}
\end{rmk}

In this way, if $TV$ is given a structure of a well graded Cacti algebra with
multiplication equal tensor product, then $H$ is simultaneously a
counitary coassociative coalgebra, and a unitary associative algebra.
Next theorem shows that $H$ is necessarily a bialgebra. In other words, 
the coproduct in $H$ is multiplicative, and hence $TV=\Om(H)$,
the Kadeishvili construction.

\begin{teo}
\label{VvsTV} \label{teo_cobar}
Let $V$ be a graded vector space, the following are equivalent

\begin{itemize}
\item[(i)] To give a well graded $\Cacti$ algebra structure on $\ra T  V$,
extending the (free) associative product in $\ra TV$ and well graded
 with respect to the
bigradng on $\ra TV$.
\item[(ii)] To give a unitary and counitary differential graded associative bialgebra structure on $H = V \oplus k 1_H$.
\end{itemize}
More precisely, the correspondence is given i the following way:

From (i) to (ii), the internal differential in $\ra T  V$, restricted to $V$ gives a differential on $V$, and the external differential induces
the restricted comultiplication in $V$, that produces the counitary
comultiplication in $H$. The action of  $B_2$ gives the
 associative product.
 
 From (ii) to (i), we only notice that $(TV,d)=\Om(H)$, and Kadeishvili
 construction gives a Cacti algebra structure that is well graded.
 \end{teo}

\begin{proof}
We ony need to prove $(i)\To(ii)$, and in this part,
we only have to check that the comultiplication in $H$ is multiplicative,
Namely.

\[
\Delta(x * y) =  (-1)^{|x_{(2)}|_i |y_{(1)}|_i} (x_{(1)} * y_{(1)}) \ot (x_{(2)} * y_{(2)})
\]
Recall the Sweedler-type notation
$\Delta x = x_{(1)} \ot x_{(2)}$.
We observe that, for $a\in TV$,
\[
d a = \Delta(a)  - [1_H, a] + d_i
\]
where $[-,-]$ is the super commutator (using the total degree) in $TH$,
so
\[
[1_H,a]=1_H\ot a-(-1)^{|a|} a\ot 1_H = 1_H\ot a-(-1)^{|a|_H+1}a\ot 1_H
\]

Now, in every $\Cacti$-algebra one has
\[
\du - \ud = \delta \udu = d \udu + \udu d
\]
because the first equality comes from computing the boundary of
the cactus $\udu$ and the second is the differential of an operation.
When evaluating in elements, using that
$d = \Delta - [1, \ ] + d_i$ one gets
\begin{eqnarray*}
 \du - \ud (x,y)  & = & d \udu (x,y) + \udu (dx,y) + (-1)^{|x|} \udu(x,dy) \\
\du - \ud (x,y)
& = &
\Delta( \udu (x,y) ) +  \udu (\Delta x,y)  + (-1)^{|x|}  \udu (x,\Delta y) \\
& -&
[1,\udu (x,y)] -  \udu ([1,x], y)  - (-1)^{|x|} \udu (x,[1,y])\\
& + &
d_i( \udu (x,y) ) +  \udu (d_i x,y)  + (-1)^{|x|}  \udu (x,d_i y) \\
\end{eqnarray*}
or equivalently (changing notation from
 $\udu$ to $*$)
\begin{eqnarray*}
-[x,y]
& = & (-1)^{|x|} \Delta(x*y) +  (-1)^{|x|+1} \Delta x * y+ x  * \Delta y \\
& + & - (-1)^{|x|} [1,x*y] - (-1)^{|x|+1} [1,x] * y -  x  *[1,y]\\
& + &   (-1)^{|x|} d_i(x*y) +  (-1)^{|x|+1} d_i x * y + x  * d_i y
\end{eqnarray*}
In order to prove what we want, we will use some identities:
\begin{eqnarray*}
d_i(x*y)
& = & d_i x * y + (-1)^{|x|_i} x  * d_i y \\
\ [x,y]
& = & (-1)^{|x|} [1,x*y] - (-1)^{|x|} [1,x] * y \\
x * [1,y]  
&=& (-1)^{[x|+1} \Delta x * y \\
x * \Delta y & = & (-1)^{|x|+1 + |x_{(2)}|_i |y_{(1)}|_i} (x_{(1)} * y_{(1)}) \ot (x_{(2)} * y_{(2)}) \\
\end{eqnarray*}
The first one is simply that the  internal differential
is a derivation for the product.

The second, comes from the identity in Cacti
\[
\udu \circ_1 \ud =
\ca{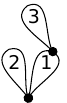}{0.9}{-3pt} + \ca{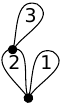}{0.9}{-3pt}
\]
because, if one evaluate this in elements, we get that $*$
verifies a left distributive law with respect to tensor product:
\begin{eqnarray*}
(a \ot  b)* c  & = & a \ot (b*c) + (-1)^{|b|(|c|+1)} (a*c) \ot b \\ 
& = & (a*1) \ot (b *c) + (-1)^{|b|(|c|+1)} (a*c) \ot (b*1)  \\
\end{eqnarray*}
and this implies immediately the equation (considering $a=1_H$, $b=x $ and $c = y$).

The two last equations have terms of the form $a * (b \ot c)$ (on their
left and side). The central idea is that, in any $\Cacti$-algebra,
even thought $*$ is not distributive on the right with $\ot$, the failure of this
is given by the boundary of $B_3$. The hypothesis of well graded
allow us to control it. In this way, we obtain that
$a * (b \ot c)$ has to be the diagonal action.
In order to see this, we calculate $\delta B_3$:
\[
\delta  \ca{12131}{0.6}{-5pt}  = \ca{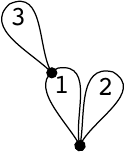}{0.6}{-5pt} - \ca{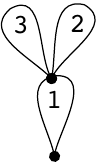}{0.6}{-5pt} + \ca{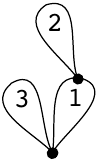}{0.6}{-5pt}
\]
and when we evaluate in elements $x,y,z \in V$ we have
\begin{eqnarray*}
\delta \ca{12131}{0.6}{-5pt} (x,y,z) & = &  \big(  \ca{2131}{0.6}{-5pt} - \ca{1231}{0.6}{-5pt} + \ca{1213}{0.6}{-5pt} \big) (x,y,z) \\
& = & 
 (-1)^{|x||y|+|x|+|y|} y \ot (x * z)  \\
&-&(-1)^{|x|}   x * (y \ot z) \\
&+&(-1)^{|x|}   (x * y) \ot z \\   
\end{eqnarray*}
But also $\delta B_3= d B_3 - B_3 d$ in $\ra T V$, so
\begin{eqnarray*}
(\delta B_3) (x,y,z) & = &
d ( \underbrace{B_3 (x,y,z)}_{=0}) - B_3 (d x, y,z)  \\
& + &(-1)^{|x|} \underbrace{B_3 (x,d y,z)}_{=0} -  (-1)^{|x|+|y|} \underbrace{ B_3(x,y, d z)}_{=0}  \\
\end{eqnarray*}
(the vanishing terms are due to the well grading hypothesis)
So,
\[
- B_3 (d x, y,z) =
 (-1)^{|x||y|+|x|+|y|} y \ot (x * z)
-(-1)^{|x|}   x * (y \ot z)
+(-1)^{|x|}   (x * y) \ot z \\ 
\]
Now, for elements in tensorial degree two $\mbf x = x_1 \ot x_2$,
the cactus $B_3$ acts by
\[
B_3(\mbf x,y,z) = B_3(x_1 \otimes x_2,y,z) = (-1)^{|x_2|+|y|+|x_2||y|} (x_1 * y) \otimes (x_2 * z)
\]
because in $\Cacti$ we have
\[
\ca{12131}{0.7}{-2pt} \circ_1 \ud = \ca{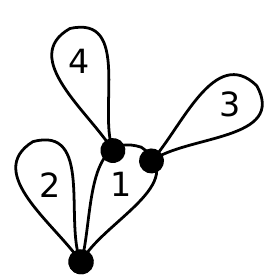}{0.5}{-2pt}  + \ca{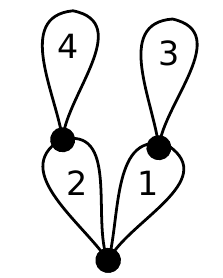}{0.5}{-2pt} + \ca{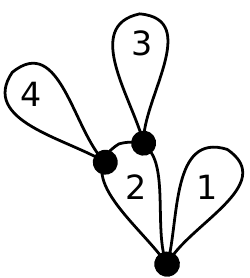}{0.5}{-2pt} 
\]

where only the second term acts non trivially in
 $V^{\otimes 4}$.

Using this identity for $\mbf x = dx$, recall
\[
dx = \Delta(x)  - [1_H , x] =  x_{(1)} \ot x_{(2)} - 1_H \ot x + (-1)^{|x|} x \ot 1_H
\]
one has
\begin{eqnarray*}
 B_3 (d x, y,z)
& = & B_3(x_{(1)} \ot x_{(2)} - 1_H \ot x + (-1)^{|x|} x \ot 1_H,y,z) \\
& = & (-1)^{|x_{(2)}|+|y|+|x_{(2)}||y|} (x_{(1)} * y) \otimes (x_{(2)} * z) \\
& &-  (-1)^{|x|+|y|+|x||y|} y \ot (x * z) \\
&  &+ (-1)^{|x|+1+|y|+1|y|} (x*y) \ot z \\ 
& = & (-1)^{|x_{(2)}|+|y|+|x_{(2)}||y|} (x_{(1)} * y) \otimes (x_{(2)} * z) \\
&  &- (-1)^{|x|+|y|+|x||y|} y \ot (x * z) \\
&  & -(-1)^{|x|} (x*y) \ot z \\
\end{eqnarray*}
where from one gets the equation
\[
x*(y\otimes z) =
(-1)^{|x| + |x_{(2)}|+|y|+|x_{(2)}||y|} (x_{(1)} * y) \otimes (x_{(2)} * z)
\]
namely, the diagonal action.

From this general equation, using $[1_H\ot y]$ instead of $y\ot z$,
we deduce
\[
x * [1,y]  
= (-1)^{[x|+1} \Delta x * y
\]
And replacing again $y \ot z$ by $\Delta y = y_{(1)} \ot y_{(2)}$ (
and of course taking into account the signs, noticing that
if  $v \in V$ then $|v|_{tot} = |v|_i + 1$):
\[
  x * (\Delta y) 
  =
(-1)^{|x|+1 + |x_{(2)}|_i |y_{(1)}|_i} (x_{(1)} * y_{(1)}) \ot (x_{(2)} * y_{(2)}) 
\]
that is precisely the last thing what we needed to verify.
\end{proof}

\begin{example} \label{ex_Lambdag} Let $\mathfrak g$
be a  Lie algebra and consider $H=U(\mathfrak g)$, and as always 
 $V=\ra U( \lie g) = \Ker(\epsilon : U(\mathfrak g)\to k)$, then
 the cohomology of $(\ra TV,d)$ is
$$ H^\bullet \left( \ra T V \right) \simeq \Lambda^\bullet\mathfrak g$$
(where here $\Lambda \mathfrak g$ is the non unital exterior algebra
in  $\mathfrak g$). Even more, in degree one, the Lie bracket in
 $H^1(\ra T V,d)$ is the commutator of the primitive elements in
  $U\mathfrak g$, namely, the Lie bracket in $\mathfrak g$.
 
  Since
$\Lambda^\bullet\mathfrak g$ is generated (as associative algebra)
in degree one, the Gerstenhaber structure is determined by the bracket
in this degree. So we get the standard Gerstenhaber algebra structure
in $\Lambda^\bullet\mathfrak g$
from the $\Cacti$-algebra in $ \ra T V$. In other words,
the Gerstenahaber algebra structure in
 $\Lambda^\bullet\mathfrak g$ lifts to a well graded
 $\Cacti$-algebra structure in  $\ra T V = \ra T \left( \ra U(\mathfrak g)\right)$.

As a sub example, if $W$ is any vector space and $\mathfrak g =
 \mathrm{Lie}(W)$ is the free Lie algebra on $W$, then
 $\Lambda^\bullet\mathfrak g = \Lambda^\bullet \mathrm{Lie}(W)$ is
 the free Gerstenhaber algebra in $W$. Again this structure
 lifts to a (well graded)  $\Cacti$-algebra structure in $\ra{ T} V = 
\ra T  \left(\ra{U ( \mathrm{Lie} (W))} \right) = \ra T \ \ra{T W}$,
in the sense that its $\Cacti$-algebra structure induces the
Gerstenhaber algebra structure on its homology.
\end{example}

\section{Morphisms and well gradings}

Recall the notation, for a bigraded algebra
 $T=\bigoplus_{p,q}T^{p,q}$
 \[T^n:=\bigoplus_{q\in\Z}T^{n,q}\]
Next Lemma is relatively simple to proof, but is the key point
of our main result. It formalizes the fact that in $TV$, all
the
$\Cacti$-algebra structure depends on $B_2$, the differential, and the associative product.

\begin{lema}\label{lemamorfismo}
Let  $T$ and $C$ be two well graded Cacti algebras, and $f:T\to C$ a
linear transformation, that is homogeneous with respect to the bigrading.
If we assume that
 \begin{itemize}
  \item $T$ is generated by $T^1$ as associative algebra (in
  particular $T=\oplus_{n\geq 1} T_n$),
  \item $f$ is a morphism of associative algebras,
\item $f(dt)=df(t)$ for all  $t\in T^1$,
\item $f|_{T^1}:(T^1,*)\to (C^1,*)$ is a morphism of associative
algebras,
  \end{itemize}
then $f$ is a morphism of  Cacti-algebras.
\end{lema}

\begin{proof}
Let us denote by $\cup$  the associative product given by $C_2$ (in
$T$ and in $C$). In an analogous way to theorem~\ref{VvsTV}, the
signs are given by the Koszul rule, but in this proof it is not necessary to
make them explicit, so we will omit then for clarity.

The proof consist in the following reductions:

\begin{enumerate}
\item  If $f(B_2(\mbf x,\mbf y))=
B_2(f\mbf x,f\mbf y))$, then
$f(M,\mbf x_1,\dots, \mbf x_n)
=f(M,f\mbf x_1,\dots,f \mbf x_n)$
 for every cactus $M$.

\begin{proof}
Since $\Cacti$ is generated by $C_2$ and $B_m$ ($m\geq 2$), it
is enough to see that
$f$ commutes with this operations. Notice that $f$ is a morphism
of associative algebras by assumption.
In order to reduce from $B_m$ to $B_2$, we proceed by induction
in the external degree. Recall the identity
\[
\raisebox{-5pt}{\includegraphics[scale=0.7]{Bm.pdf}} \circ _1
 \raisebox{-3pt}{\includegraphics[scale=0.7]{ca12.pdf}} = \sum_k
  \raisebox{-5pt}{\includegraphics[scale=0.7]{Bm_1_C2_k.pdf}}
\]
If we want to compute
$B_m(\mbf x, \mbf y_1, \dots, \mbf y_{m-1})$, with
 $\mbf x \in T ^{p,\bullet }$, the well grading implies that
  the non-trivial terms are only with $p \ge m$. Considering elements
  $\mbf x =  x_1 \cup \mbf x'$ with $x\in T^1$ and
$\mbf  x'\i T^{m-1}$ (this is possible because we assume $T$
is generated by $T^1$) we have
\[
B_m\big( (x_1  \cup \mbf x'), \mbf y_1, \dots, \mbf y_{m-1} \big)
=\]
\[
\sum_{k=1}^m \pm B_k(x_1, , \mbf y^1, \dots, \mbf y^{k-1})
\cup   B_{m-k+1}(\mbf x', \mbf y^k, \dots, \mbf y^{m-1}) 
\]
and because of the well-grading ($|x_1|_e =1$, so every term is zero
 except two of them)
\begin{eqnarray*}
B_m\big(C_2 (x_1, \mbf x'), \mbf y_1, \dots, \mbf y_m \big)
& = &
\pm C_2 (x_1, B_m (\mbf x', \mbf y^1, \dots, \mbf y^{m-1}) )  \\
& & \pm C_2 (B_2(x_1,y_1) , B_{m-1} (\mbf x', \mbf y^2, \dots, \mbf y^{m-1}) ) \\
\end{eqnarray*}
where the first term has $|x'|_e < |x|_e$, and the second
is written using  $B_2$ y $B_{m-1}$. Hence,  $f$ commutes
with all $B_m$ if it does with $B_2$.
\end{proof}

\item  If $fB_2(x,\mbf y)=
B_2(fx,f\mbf y)$ for all $x\in T^1$, $\mbf y\in T$, then
 $fB_2(\mbf x,\mbf y)=
B_2(f\mbf x,f\mbf y)$ for all $\mbf x,\mbf y\in T$,
\begin{proof}
If $\mbf x = x_1 \cup \dots \cup x_r$, since $B_2$ distribute the
$\cup$-product in the first variable, we have
\[
B_2(\mbf x , \mbf y ) =
 \sum_{k=1}^r \pm \ x_1 \cup \dots B_2(x_k, \mbf y) \dots \cup x_r
\]
so the claim follows.
\end{proof}

\item
If $fB_2(x,y)=
B_2(fx,fy)$ for all $x,y\in T^1$ (which is true by assumption),
then
 $fB_2(x,\mbf y)=
B_2(f x,f\mbf y)$ for all $x\in T^1$, $\mbf y\in T$,
\begin{proof}
Let
 $\mbf y=\mbf y'\cup \mbf y' \in T$, notice that the
 external degree of $\mbf y'$ and $\mbf y''$ are both strict less than
 the degree of $\mbf y$.
For $x\in T^1$, we compute
\[
B_2(x,\mbf y)=
B_2(x,\mbf y'\cup \mbf y'')=
B_2\circ_2C_2(x,y',y'')
\]
\[
=
\pm B_2(x,\mbf y')\cup \mbf y''
\pm \mbf y' B_2(x,\mbf y'')+(\delta B_3)(x,\mbf y',\mbf y'')
\]
Notice that (due to the well grading and the fact that $x\in T^1$):
\begin{eqnarray*}
(\delta B_3)(x,\mbf y',\mbf y'') & = &
d(B_3(x,\mbf y',\mbf y'')
+ B_3(dx,\mbf y',\mbf y'')
\pm B_3(x,d\mbf y',\mbf y'')
\pm B_3(x,\mbf y',d\mbf y'') \\
& = & 
B_3(dx,\mbf y',\mbf y'') \\
\end{eqnarray*}
Now, since $dx\in T^1\oplus T^2$ and $T$ is generated by $T^1$ as
associative algebra, we can write
\[
dx=d_ix+\sum x_1\cup x_2\]
and so
\begin{eqnarray*}
B_3(dx,\mbf y',\mbf y'')
& = & B_3(d_ix,\mbf y',\mbf y'')+ B_3(x_1\cup x_2,\mbf y',\mbf y'') \\
& = & B_3(x_1\cup x_2,\mbf y',\mbf y'') \\
& = & (B_3\circ_1 C_2)(x_1,x_2,\mbf y',\mbf y'') \\
& = & \pm B_3(x_1,\mbf y',\mbf y'')\cup x_2 \\
& & \pm B_2(x_1,\mbf y')\cup B_2(x_2,\mbf y'') \\
& & \pm x_1\cup B_3(x_2,\mbf y',\mbf y'') \\
& =& \pm B_2(x_1,\mbf y')\cup B_2(x_2,\mbf y'') \\
\end{eqnarray*}
(we have used again that well grading hypothesis and the fact that
$d_ix$,  $x_1$ and $x_2$ belong to $T^1$).

We conclude
\[
B_2(x,\mbf y)
 =
\pm B_2( x,\mbf y) \cup \mbf y''
\pm \mbf y' \cup B_2(x,\mbf y'')
\pm B_2(x_1,\mbf y')\cup B_2(x_2,\mbf y'')
\]

With this in mind, we compute
\[
f(B_2(x,\mbf y))
=
f\left(
\pm B_2(x,\mbf y')\cup \mbf y''
\pm \mbf y' B_2(x,\mbf y'')
\pm B_2(x_1,\mbf y')\cup  B_2(x_2,\mbf y'')
\right)
\]
and since $f$ commutes with $\cup$
\[=
\pm f B_2(x,\mbf y')\cup f\mbf y''
\pm f\mbf y' \cup fB_2(x,\mbf y'')
\pm fB_2(x_1,\mbf y')\cup fB_2(x_2,\mbf y'')
\]
Because $\mbf y'$ and $\mbf y''$ have strict less external degree than
 $\mbf y$, we may assume  inductively that $f$ preserves the operation
  $B_2(x,-)$ in those degrees, and so the above formula is equal to
 \[ =
\pm  B_2(fx,f\mbf y')\cup f\mbf y''
\pm f\mbf y' \cup B_2(fx,f\mbf y'')
\pm B_2(fx_1,f\mbf y')\cup B_2(fx_2,f\mbf y'')
\]
and since $f$ preserves degrees, and also $C$ is well graded,
the arguments used to eliminate the terms of type
 $B_3(x,-)$ can also be used for $B_3(fx,-)$, and we conclude
 \[=
\pm  B_2(fx,f\mbf y')\cup f\mbf y''
\pm f\mbf y' \cup B_2(fx,f\mbf y'')
\pm B_2(fx_1\cup fx_2,f\mbf y'\cup f\mbf y'')
\]
\[=
\pm  B_2(fx,f\mbf y')\cup f\mbf y''
\pm f\mbf y' \cup B_2(fx,f\mbf y'')
\pm B_2(f dx,f(\mbf y'\cup \mbf y''))
\]
Finally because $f$ commutes with the differential in $T^1$,
we have that  $fdx=d(fx)$ and so
\[=
B_2(fx,f\mbf y' \cup f\mbf y'')=B_2(fx,f\mbf y)
\]
\end{proof}
\end{enumerate}
Since the requirement of the last reduction holds by assumption
of the lemma, we have finished the proof.
 \end{proof}

As immediate corollary, we see that Theorem~\ref{VvsTV} actually
gives an equivalence of categories between d.g. bialgebras and Cacti
algebras that are well graded and freely generated in external degree one:

\begin{coro} \label{coro_morfismos}
Let $H$ and $H'$ be two (d.g.) unitaries and counitaries bialgebras,
and endow $\Omega H = \ra T \ \ra H$ and
$\Omega H' = \ra T \ \ra H'$ with its natural Cacti algebra structure, then
\[
\Hom_{\Cacti}( \Omega H, \Omega H') \xto
\cong \Hom_{\text{d.g.bialg}} (H, H' ) 
\]
\end{coro}
\begin{proof}
We only remark that both $\Omega H$ and $\Omega H'$ are well
graded and generated in external degree one, so the Lemma above applies.
\end{proof}

\begin{rmk} The Cacti algebra structure on $\Om H$ is unique if
one requires well grading, and that the operation $B_2$ restricted to
$H$ agree with
the product of $H$. This is true because if
$\wt { \Om H }$ is equal to $\Om(H)$ as d.g. algebras, but with eventually different
Cacti algebra structure with this properties, then
the identity map $\Om H  \to \wt {  \Om H}$ verifies the hypothesis of the
above lemma, and hence it must be a Cacti-algebra morphism.
\end{rmk}

Next theorem is a continuation of the dictionary between Cacti
 axioms and bialgebra axioms, Before present it, we recall a standard
definition of a modulo-algebra.

\begin{defi} \label{def_H_modulo}
Let $A$ be a unital associative algebra and $H$ a unitary and counitary
bialgebra. We say that
$\rho:H\ot A\to A$ is an {\bf
$H$-module algebra structure} on $A$
if it makes $A$ into an $H$-modulo but also satisfying the property that
the multiplication map
\[
m_A:A\ot A\to A
\]
is $H$-linear (with the diagonal action on $A\ot A$).

In case $A$ is a d.g. algebra and $H$ a d.g. bialgebra,
the $H$-module algebra structure is called differential if
\[
d(h(a))=d_H(h)(a) +(-1)^{|h|}h(d_A(a))
\]
or equivalently if the map
\[
\rho:H\ot A\to A
\]
is a morphism of complexes.
\end{defi}

\begin{teo}\label{teomorfismo}
Let $A$ be a d.g. unital associative algebra
and $H$ a d.g. unital and counital bialgebra. Then there exists
a 1-1- correspondence between Cacti algebra
morphism $\Om(H)\to  C^\bullet(A)$ and
differential $H$-module algebra structures on $A$.
The correspondence is given by restriction:
\[\Hom_{\Cacti}(\Om H ,C^\bullet (A))\to \Hom_{d.g.alg}(\ra H,\End(A))\cong  \Hom_{d.g.alg_1}( H,\End(A))\cong
\Hom(H\ot A,A)
\]
and in the other direction, if $\rho:H\to \End(A)$, $x\mapsto
\rho_x$, the map
$\Om(H)\to C^\bullet(A)$ is
given by
\[
TV\ni x_1\ot\cdots\ot x_n\mapsto
\big(a_1\ot\cdots\ot a_n\mapsto
\rho_{x_1}(a_1)
\cdots \rho_{x_n}(a_n)\big)
\]
  \end{teo}
In this theorem,  $d.g.alg$ means non necesarily unital
 differential graded associative algebras, and
 $d.g.alg_1$ are the $d.g.alg$ maps that also preserve the unit.

 \begin{proof}
Since   $\Omega H$ and  $C(A)$ are
 both well graded Cacti algebras, we can use
 Lemma \ref{lemamorfismo}.
Then, a morphism $f:\Omega H \to C(A)$ is the same as a
d.g. algebra morphism such that its restriction on
elements of external degree one (i.e. to elements of $H$)
is multiplicative with respect to the operation $*$. This produces
a morphism
\[
\rho:=f|_V:V\to \End(A)
\]
where $V=\ra H=\Ker(\epsilon)$.
This shows that morphisms whose restriction are $*$-multiplicative
are the same as (non unital)  $V$-modulo structures
 on $A$, that is the same as unital $H$-module structures on $A$.

Notice that given an $H$-module structure $\rho:H\to \End(A)$,
the restriction to $V$ produces a multiplicative map $V\to\End(A)$.
Then, the universal property of the tensor algebra gives a 
multiplicative map $\wh\rho:(TV,\ot )\to (C(A),\cup)$.
The theorem follows if we show that
``$\wh\rho$ commutes with the differential if and only if
the $H$-module structure is a (differential) $H$-module
{\em algebra}
structure''.

Let us denote, for $h\in H$ y $a\in A$,
\[
h(a):=(\rho(h))(a)
\]
When computing the Hochschild boundary of $\rho(h)$
we get
\[
(d_e\rho(h))(a\ot b)
=-a h(b)+h(ab)-h(a)b
\]
On the other hand, the internal differential is
\[
(d_i\rho(h))(a)=d(h(a))-(-1) ^{|h|}h(d(a))
\]
Because $d=d_e+d_i$ and their bidegrees are different, the equality
\[
d\rho(h)
=\wh\rho dh
\]
is equivalent to two equations
\[
d_e\rho(h)
=\wh\rho d_eh,\
d_i\rho(h)=
\rho d_ih
\]
The equation with $d_e$ tell us that $A$ is an $H$-module algebra,
because
\[
(\wh\rho d_ih)(a\ot b)
=(\wh\rho (\Delta h-1_H\ot h-h\ot 1_H))(a\ot b)
\]\[
=(\wh\rho (h_1\ot h_2-1_H\ot h-h\ot 1_H))(a\ot b)
\]\[=h_1(a)h_2(b)-h(a)b-ah(b)
\]
and hence
\[
\partial\rho(h)=\wh\rho d_i h
\iff h(ab)=h_1(a)h_2(b)\]

And the equation with $d_i$ says
\[
(d_i\rho(h))(a)=d_A(h(a))-(-1) ^{|h|}h(d_A(a))=d_H(h)(a)
\]
namely, the d.g. condition for $\rho$.
\end{proof}

An immediate consequence is the following
\begin{coro}
Let $H$ be a bialgebra and $A$ an $H$-module algebra
with structure map
$H\ot A\xto \rho A$. Then  $\rho$ induces a Gerstenhaber algebra map
\[
H^\bullet(\Omega H,d)\to HH^\bullet(A)
\]
\end{coro}

\subsection*{Examples}

Let $\mfk g$ be a Lie algebra and $H=U(\mfk g)$. If $A$ is an associative
algebra, then an $H$-module algebra map is the same as an action of
$\mfk \g$ by derivations. If one takes
$\mfk g=\Der(A)$, then the morphism
$\Omega H \to C(A)$ induces a map on homology
\[
\Lambda ^\bullet\Der(A)\to HH ^\bullet(A)
\]
whose image is the associative subalgebra of $HH^\bullet(A)$
generated by derivations.
This example shows that the map from Theorem \ref{teomorfismo}
is in general non trivial.

But it can happen that a bialgebra $H$ has no primitive elements
but non-trivial cohomology. This will produce maps with no derivations
in its image, but giving elements of higher (cohomological) degree.

We present a minimal example of this situation.

Let $H=k1\oplus kx\oplus kg\oplus kgx$ be the Sweedler or Taft algebra
of dimension 4, that may be describe in terms
of generators and relations as the $k$-algebra generated by
$x$ and $g$ with relations
\[
 x^2=0,\  g^2=1, \ xg=-gx
\]
(we assume characteristic different from 2). This algebra is a bialgebra
with comultiplication determined by
\[
\Delta(g)=g\ot g \quad \Delta x=x\ot 1+g\ot x
\]
This algebra has no primitive elements, so $H ^1(\Omega H)=0$,
but a direct computation shows that the (class of the) element
 $xg\ot x$
generates $H^2(\Omega H)$ over $k$.
A less direct computation shows that
$H ^\bullet(\Omega H)$ is a polynomial ring on one variable, with
 generator in degree two (given by this element).
Next, we include the verification of this fact, that follows from
these
three items:
\begin{itemize}
 \item $H\cong H^*$ as Hopf algebras, for instance, take the elements in
  $H^*$ defined by
  \[
  \wh g:\left\{
  \begin{array}{ccr}
   1&\mapsto& 1\\
   g&\mapsto& -1\\
   x&\mapsto& 0\\
   xg&\mapsto& 0  
  \end{array} 
  \right.
\hskip 2cm  \wh x:\left\{
  \begin{array}{ccc}
   1&\mapsto& 0\\
   g&\mapsto& 0\\
   x&\mapsto& 1\\
   xg&\mapsto& 1  
  \end{array} 
  \right.
\]
one can easily verify that
 $\wh g ^2=\epsilon$, $\wh g\wh x=-\wh x\wh g$, $\wh x^2=0$.
For that reason, we have an isomorphism
 \[
  H^\bullet(\Omega H) =\Ext^\bullet_{H^*}(k,k)
  \cong \Ext^\bullet_H(k,k)
    \]
 
 \item Also, $H=(k[x]/x^2)\# k \ZZ_2$, so one can compute
  $\Ext$ with the formula
  \[
  \Ext_H^\bullet(k,k)=\Ext^\bullet_{k[x]/x^2}(k,k)^{\ZZ_2}
  \]
 (see for instance \cite{Ste95})
  \item
$\Ext^\bullet_{k[x]/x^2}(k,k)$ is a polynomial ring in one variable of
degree one (this is the easiest example of classical quadratic
 Koszul algebra), call the generator $D$.
 There re two possibilities:
the action of the generator of $\ZZ_2$ is trivial in $D$, or it
acts by $D\mapsto -D$. In the first case it should be
  $\Ext_{k[x]/x^2}(k,k)^{\ZZ_2}=k[D]$, while in the second
  it should be  $\Ext_{k[x]/x^2}(k,k)^{\ZZ_2}=k[D^2]$.
But in $H$ there are no primitive elements, so
$H^1(\Omega H)=0$ and only the second possibility can be true.
\end{itemize}

A consequence of this commutation is that the  Gerstenhaber
bracket (in cohomology) of the generator with itself is trivial,
just by degree considerations. This implies that in any
$H$-module algebra $A$,  the bilinear map given by
\begin{eqnarray*}
\Psi:A^{\ot 2} & \to & A \\
a\ot b & \mapsto & xg(a)x(b) \\
\end{eqnarray*}
is an integrable  2-cocycle in the sense that $[\Psi,\Psi]=0$.

We finally recall that the data of an $H$-module algebra structure
on $A$ is the same as a $\ZZ_2$-grading
(given by the eigenvectors of eigenvalues $\pm 1$ of $g$,
we assume char$\neq $2) and a square zero
super-derivation (with respect to that grading), because the
general formula
\[
h(ab)=h_1(a)h_2(b)
\]
for  $h=x$ says  (if $a$ is homogeneous):
\[
x(ab)=x(a)b +g(a)x(b)=
x(a)b +(-1)^{|a|}ax(b)
\]
In that way, every square zero super derivation $x$ in $A$
gives an unobstructed formal deformation of $A$.

We finish by collecting some general information on Hopf
algebras and its cohomology

\begin{enumerate}
\item If $H$ is finite dimensional bialgebra, then
\[
H^\bullet(\Om(H))=\Ext^\bullet_{H^*}(k,k)
=H^\bullet(H^*,k)\subset HH^\bullet(H^*)\]
These equalities are immediate from the definition if one uses
the standard complex for solving $H^*$ as $H^*$-bimodule
when computing Hochschild cohomology. The last inclusion
was proved (to be a split inclusion)
in \cite{FS}, by  giving a specific map at the level of complexes, that
reserves the cup product and $i$-th compositions. Now this map
can be interpreted from the fact that any finite dimensional
bialgebra is an $H^*$-module algebra. The finite dimensional
hypothesis is only
needed for $H^*$ to be a bialgebra as well. We remark that the notation $*$
that we used 
for the product in a general bialgebra  is motivated by this example,
since the $B_2$ operation in $\Ext_H^\bullet(H,k)$ comes from the convolution product in $H^*$.

\item If $H$ is any bialgebra and $H^\prime$ is a bialgebra in
duality with $H$, namely there is a pairing $(-,-):H\ot H^\prime \to k$
satisfying
\[
(\Delta a, x\prime\ot y^\prime)=(a,x^\prime\ot y^\prime)
\]
and
\[( a\ot b, \Delta x)=(ab,x^\prime)
\]
then $H$ is an $H^\prime$-module algebra.
\item If $A$ is an $H$-comodule algebra, that is, it
is given a comodule structure map
\[
A\to A\ot H
\]
such that the multiplication $m_A:A\ot A\to A$ is $H$-colinear,
and $H^\prime$ is in duality with $H$, then $A$ is an
$H^\prime$-module algebra.
Geometrical examples come in this way:
if $X$ and $G$ are affine algebraic varieties and $G$ is an algebraic
group, to have an algebraic action  of $G$ on $X$ is the same
as a comodule algebra structure $\O_X\to \O_X\ot \O_G$.
If $\mfk\g$ is the Lie algebra associated to the algebraic
group $G$, then $U\g$ is in duality with $\O_G$, and hence
$A$ is an $U\g$-module algebra.

\item If $G$ is a discrete group and $H=k[G]$, an $H$-module
algebra structure on $A$ is the same as a $G$-grading,
 but this gives nothing interesting because $k[G]$ is cosemisimple.

\item The general Taft algebra $H=T_m$:
This algebra is generated by $x,g$
with relations
\[
g^ p = 1;\ x^m = 0;\  gx = \xi xg\]
where $\xi$ is a primitive $m$-th root of unity.
The comultiplication given by
$\Delta g=g\ot g$,
$\Delta x=x\ot g+1\ot x$,
In order to compute the cohomology, one can see
$H\cong (k[x]/x^m)\#k[G]$ with $G=\langle g: g^m=1\rangle$,
and so $H^*$ is also of the form
$H^*\cong A\#k^G$. The same result in \cite{Ste95}
gives a formula
\[
H^\bullet(A\# k^G,k)=
H^0(k^G,H^\bullet(A,k))=
H^\bullet(A,k)_0
\]
where $H^\bullet(A,k)_0$ is the homogeneous component of degree
zero with respect to the $G$-grading of
$H^\bullet(A,k)$.

\item Let $g$ be a group-like element in some bialgebra $H$,
 and
denote  $u_g=g-1_H$ (notice $g\notin\Ker\epsilon$, but
$u_g\in\Ker\epsilon$). Then $d(u_g)=\Delta' u_g=u_g\ot u_g$.
If $h$ is another gouplike element and $x$ is $g$-$h$-primitive, namely
$\Delta x=g\ot x+x\ot h$, then
$dx=u_g\ot x+x\ot u_h$.
\begin{proof}
\[
d(u_g)=d(g-1_H)- 1_H\ot u_g-u_g\ot 1_H
=g\ot g-1_H\ot 1_H - 1_H\ot u_g-u_g\ot 1_H
\]
\[
=g\ot g-1_H\ot 1_H - 1_H\ot g+1_H\ot 1_H -g\ot 1_H+1_H\ot 1_H
\]
\[
=(g-1_H)\ot (g-1_H)=u_g\ot u_g
\]
The formula
$dx=u_g\ot x+x\ot u_h$ is proved in a complete analogous way, we omit it.
\end{proof}
This computation allow us to generalize the
example of the $H$-module algebra
action of the Sweedler algebra in the following way:

Let $d_1,\dots, d_n:A\to A$ be skew-derivation of an associative
algebra $A$. That means there exist $g_i$ and $h_i$
automorphisms of algebras of $A$ such that
\[
d_i(ab)=
g_i(a)d_i(b)+
d_i(a)h_i(b)\ \forall a,b\in A
\]
Let $f:A^{\ot n}\to A$ be defined as
\[
f(a_1,\dots,a_n)=d_1(a_1)\cdots d_n(a_n)
\]
If, in addition,
$g_0=g_{n+1}=\id$ and $h_i=g_{i+1}$  for all $i=1,\dots, n-1$,
then $f$ is a Hochschild  $n$-cocycle, coming from
$\Om(H)$ for some bialgebra $H$.
\begin{proof}
Let us consider the free algebra generated by $x_i: i=1,\dots, n$
and $G_i:i=0,\dots, n+1$, with comultiplication
determined by
\[
\Delta G_i=G_i\ot G_i;\ \Delta
x_i=G_i\ot x_i+x_i\ot G_{i+1}
\]
and define the $H$-module structure on $A$ by
\[
x_i(a)=d_i(a),\ G_i(a)=g_i(a)\]
where, by notation $g_{n+1}=h_n$. Then $A$ is an $H$-module algebra.
We need to check that
$\om:=x_1\ot\cdots\ot x_n\in\Om(H)$ satisfies $d\om=0$.
But this is easy because
\[
d\om=
d(x_1\ot\cdots\ot x_n)=\sum_{i=1}^n
(-1)^{i+1}x_1\ot\cdots \ot d(x_i)\ot
\cdots \ot x_n
\]
\[
=\sum_{i=1}^n
(-1)^{i+1}(x_1\ot\cdots \ot u_{G_i}\ot x_i\ot
\cdots \ot x_n+
x_1\ot\cdots \ot x_i\ot u_{G_{i+1}}\ot
\cdots \ot x_n)
\]
and all terms cancel telescopically except the first and the last:
\[
=u_{G_0}\ot x_1\ot\cdots  \ot x_n
+(-1)^{n-1} x_1\ot\cdots  \ot x_n\ot  u_{G_{n+1}}
\]
but $u_{G_0}
=u_{G_{n+1}}
=u_{id}=0$, so $d\om=0$, and hence $\partial f=0$ in $C^\bullet (A)$.
\end{proof}
We remark that also the other Cacti operations that one may
do with $f$ in $C^\bullet(A)$ may also be done in $\Om(H)$.
\end{enumerate}
It would be interesting to know, given an associative algebra 
$A$, whether or not any class in $HH ^\bullet(A)$ comes from
an element in $H ^\bullet(\Om(H))$, for some bialgebra $H$ acting
on $A$, making $A$ into an $H$-module algebra.

\bibliographystyle{halpha}
\addcontentsline{toc}{chapter}{References}
{ \small \hyphenpenalty=10000

 \bibliography{goodgraded}

\begin{thebibliography}{Kau07}

\bibitem[BF04]{BF04}
Clemens Berger and Benoit Fresse.
\newblock Combinatorial operad actions on cochains.
\newblock {\em Math. Proc. Cambridge Philos. Soc.}, 137(1):135--174, 2004.

\bibitem[FS04]{FS}
M~Farinati and A~Solotar.
\newblock G-structure on the cohomology of hopf algebras.
\newblock {\em Proc. Amer. Math. Soc.}, 132:2859--2865, 2004.

\bibitem[Kad05]{Kad05}
Tornike Kadeishvili.
\newblock On the cobar construction of a bialgebra.
\newblock {\em Homology, Homotopy and Applications}, 7(2):109--122, 2005.

\bibitem[Kau07]{Kau07}
Ralph~M. Kaufmann.
\newblock On spineless cacti, {D}eligne's conjecture and {C}onnes-{K}reimer's
  {H}opf algebra.
\newblock {\em Topology}, 46(1):39--88, 2007.

\bibitem[MS02]{McS02}
James~E. McClure and Jeffrey~H. Smith.
\newblock A solution of {D}eligne's {H}ochschild cohomology conjecture.
\newblock In {\em Recent progress in homotopy theory ({B}altimore, {MD},
  2000)}, volume 293 of {\em Contemp. Math.}, pages 153--193. Amer. Math. Soc.,
  Providence, RI, 2002.

\bibitem[Ste95]{Ste95}
D.~Stefan.
\newblock Hochschild cohomology on {H}opf {G}alois extensions.
\newblock {\em Journal of Pure and Applied Algebra}, 103(2):221 -- 233, 1995.

\bibitem[VG95]{GV95}
A.~A. Voronov and M.~Gerstenkhaber.
\newblock Higher-order operations on the {H}ochschild complex.
\newblock {\em Funktsional. Anal. i Prilozhen.}, 29(1):1--6, 96, 1995.

\bibitem[You13]{justin}
J~Young.
\newblock Brace bar-cobar duality.
\newblock 2013, arXiv:1309.2820.

\end{thebibliography}
}

\end{document}